\documentclass[11pt]{amsart}

\newtheorem{theorem}{Theorem}[section]

\newtheorem{corollary}[theorem]{Corollary}
\newtheorem{proposition}[theorem]{Proposition}

\theoremstyle{definition}
\newtheorem{definition}[theorem]{Definition}

\theoremstyle{remark}

\numberwithin{equation}{section}
\usepackage{prettyref}
\usepackage{latexsym}
\usepackage{amsmath}
\usepackage{amssymb}
\usepackage{color}
\usepackage{array}

\usepackage[
colorlinks=true,
linkcolor=blue,
 citecolor=blue,
  urlcolor=blue,
]{hyperref}

\begin{document}
\def\Xint#1{\mathchoice
  {\XXint\displaystyle\textstyle{#1}}%
  {\XXint\textstyle\scriptstyle{#1}}%
  {\XXint\scriptstyle\scriptscriptstyle{#1}}%
  {\XXint\scriptscriptstyle\scriptscriptstyle{#1}}%
  \!\int}
\def\XXint#1#2#3{{\setbox0=\hbox{$#1{#2#3}{\int}$}
    \vcenter{\hbox{$#2#3$}}\kern-.5\wd0}}
\def\ddashint{\Xint=}
\def\avgint{\Xint-}

\title[Optimality without examples] { Optimal exponents in weighted estimates without examples }

\author{Teresa Luque}
\address{Departamento de An\'alisis Matem\'atico,
Facultad de Matem\'aticas, Universidad de Sevilla, 41080 Sevilla,
Spain} 
\email{tluquem@us.es}

\author{Carlos P\'erez}
\address{Departamento de An\'alisis Matem\'atico,
Facultad de Matem\'aticas, I.M.U.S., Instituto de Matem\'aticas de la Universidad de Sevilla, Universidad de Sevilla 41080 Sevilla,
Spain} 
\email{carlosperez@us.es}

\author{Ezequiel Rela}
\address{Departamento de An\'alisis Matem\'atico,
Facultad de Matem\'aticas, Universidad de Sevilla, 41080 Sevilla,
Spain} 
\email{erela@us.es}

\thanks{The first author is supported by the Spanish Ministry of Science and Innovation grant MTM2012-30748,
the second and third authors are also supported by the Junta de Andaluc\'ia, grant FQM-4745.}



\begin{abstract} 
We present a general approach for proving the optimality  of  the exponents on weighted estimates. We show that if an operator $T$ satisfies a bound like 
\begin{equation*}
\|T\|_{L^{p}(w)}\le c\, [w]^{\beta}_{A_p} \qquad w \in A_{ p},
\end{equation*}
then the  optimal lower bound  for $\beta$ is closely related to the  asymptotic behaviour of the unweighted $L^p$ norm $\|T\|_{L^p(\mathbb{R}^n)}$  as $p$ goes to $1$ and $+\infty$, which is related to Yano's classical extrapolation theorem. 
By combining these results with the known weighted inequalities, we derive the sharpness of the exponents, without building any specific example, for a wide class of operators including maximal-type, Calder\'on--Zygmund 
and fractional operators. In particular, we obtain a lower bound for the best possible exponent for Bochner-Riesz multipliers. 
We also present a new result concerning a continuum family of maximal operators on the scale of logarithmic Orlicz functions.
Further, our method allows to consider in a unified way maximal operators defined over very general Muckenhoupt bases.

\end{abstract}

\maketitle

\section{Introduction and statement of the main result}

\subsection{Introduction}
A main problem in modern Harmonic Analysis is the study of sharp norm inequalities for some of the classical operators on weighted Lebesgue spaces $L^p(w), \, 1<p<\infty$.  The usual examples include the Hardy--Littlewood maximal operator, the Hilbert transform and more general Calder\'on-Zygmund operators (C--Z operators). Here $w$ denotes a non--negative, locally integrable function, that is a weight. The class of weights for which these operators $T$ are bounded on $L^p(w)$ were identified in \cite{Muckenhoupt:Ap} and in the later works \cite{HMW}, \cite{CF}. This class consists of the Muckenhoupt $A_{p}$ weights defined by the condition
\begin{equation*}
[w]_{A_p}:=\sup_{Q}\left(\frac{1}{|Q|}\int_{Q}w(y)\ dy \right)\left(\frac{1}{|Q|}\int_{Q}w(y)^{1-p'}\ dy \right)^{p-1}<\infty, 
\end{equation*}
where the supremum is taken over all the cubes $Q$ in $\mathbb{R}^n$, $1<p<\infty$ and as usual $p'$ stands for the dual exponent of $p$ satisfying $1/p+1/p'=1$. 

Given any of these operators $T$, the first part of this problem is to look for quantitative bounds of the norm $\|T\|_{L^p(w)}$ in terms of  the $A_p$ constant of the weight. Then, the following step is to find the sharp dependence, typically with respect to the power of $[w]_{A_p}$. In recent years, the answer to this last question has let a fruitful activity and development of new tools in Harmonic Analysis. 

The first classical example is the case of the Hardy--Littlewood maximal function defined as 
\begin{equation*}
 Mf(x)=\sup_{x\in Q}\avgint_Q |f(y)| \ dy,
\end{equation*}
where the supremum is taken over all cubes containing the point $x$ and with sides parallel to the coordinate axes. As usual, we denote by $\avgint_A f$ the average of the function $f$ over the set $A$. It is well known that if $M$ is the maximal function, then
\begin{equation}\label{eq:buckley}
\|M\|_{L^{ p }(w)} \le c \, [w]^{   \frac{1}{p-1}  }_{A_{ p}}, \qquad w \in A_{ p},
\end{equation}
 and the exponent is sharp, namely $\frac{1}{p-1}$ cannot be replaced with $\frac{1-\varepsilon}{p-1}$, $\varepsilon>0$. This is due to Buckley \cite{Buckley}.

Similarly 
S. Petermichl showed in \cite{Petermichl:Riesz} that 
\begin{equation}\label{eq:Petermichl}
\|T\|_{L^{ p }(w)} \le c \, [w]^{ \max\{1,   \frac{1}{p-1} \}}_{A_{ p}}, \qquad w \in A_{ p}
\end{equation}
is sharp when $T$ is any Riesz transform. In each of these papers, the optimality of the exponent is  shown by exhibiting specific examples adapted to the operator under analysis. In the case of Riesz transform, the examples are specific  for the range $1<p<2$ and then, the sharpness for large $p$ is obtained by duality.

Similar weighted estimates are known to be true for other classical operators, such as commutators $[b,T]$ of C--Z operators and BMO functions, the dyadic square function $S_d$, vector valued maximal operators $\overline{M}_q$ for $1\le p,q\le \infty$, Bochner-Riesz multipliers $B^\lambda$ and fractional integrals $I_\alpha$. In the case of sharp bounds with respect to the power of the $A_p$ constant of the weight $w$, the sharpness is always proved by constructing specific examples for each operator. 

\subsection{Main results}
The main purpose of this article is to present a new approach to test sharpness of weighted estimates. We provide  a very general scheme that can be applied to most of the classical operators in Harmonic Analysis. In particular, we show that there is no need to build such examples and that the sharpness is intimately related to the unweighted $L^p$ norm behaviour of the operator $T$ as $p$ gets close to the endpoint $p=1$ and $p=\infty$.  The key ingredient is an application of the so called Rubio de Francia's iteration algorithm. This is a basic but powerful technique that was fruitful since it was first applied to factorization of weights and extrapolation. In particular, we will be using some ideas from the new proof of the extrapolation theorem from \cite{Javi-Duo-JFA} and also from \cite{CMP-Book}.

To illustrate the aim of the next definition, consider the following example.  Let $H$ be  the Hilbert transform.  Then, it is known that the size of its kernel implies that the unweighted $L^p$ norm satisfies
\begin{equation}\label{eq:endpointH}
 \|H\|_{L^p(\mathbb{R}^n)}\sim O(\frac{1}{p-1}).
\end{equation}
This condition is a particular case of the classical Yano condition related to the well known Yano's extrapolation theorem as shown in \cite{Yano} (see also \cite[p. 61, Theorem 3.5.1]{Guzman-real} for more details and \cite{Carro-JFA} for a generalization of these ideas). In particular, the above condition allows to prove endpoint boundedness properties for the operator in appropriate $L\log L$ spaces at local level. However, the relevant feature for our purpose is that the operator norm \emph{blows up} with order $1$ and no less. 
Influenced by this condition  we give a precise definition which tries to capture this \emph{endpoint order} by looking at the asymptotic behaviour of the $L^p$ norm of a general operator $T$.

\begin{definition}\label{def:orders}
Given a bounded operator $T$ on $L^p(\mathbb{R}^n)$ for $1<p<\infty$, we define $\alpha_T$ to be the ``endpoint order" of $T$ as follows:
\begin{equation}\label{eq:endpoint1}
\alpha_T=:\sup\{\alpha\ge 0: \forall \varepsilon>0, \limsup_{p \to 1 }  (p-1)^{\alpha-\varepsilon} \|T\|_{L^p(\mathbb{R}^n)} =\infty\}.
\end{equation}
The analogue of \eqref{eq:endpoint1} for $p$ large is the following. Let $\gamma_T$ be defined as follows
\begin{equation}\label{eq:endpointINF}
\gamma_T=:\sup\{\gamma\ge 0: \forall \varepsilon>0, \limsup_{p \to \infty }  \,\frac{\|T\|_{L^p(\mathbb{R}^n)}}{p^{\gamma-\varepsilon}} =\infty\}.
\end{equation}
\end{definition}
This definition may have appeared previously in the literature but we are not aware of it.

Now we can state our main result.

\begin{theorem} \label{thm:AbstractBuckley}
Let $T$ be an operator (not necessarily linear). Suppose further that 
for some $1<p_0<\infty$ and for any $w \in A_{ p_0}$
\begin{equation}\label{eq:weighted}
\|T\|_{L^{ p_{0} }(w)} \le c \, [w]^{\beta}_{A_{ p_0}}. 
\end{equation}
Then  $\beta\ge \max\left \{\gamma_T;\frac{\alpha_T}{p_0-1}\right \}$.
\end{theorem}

The novelty here is that we can exhibit a close connection between the weighted estimate and the unweighted behaviour of the operator at the endpoints $p=1$ and $p=\infty$. This result can be applied to the known inequalities \eqref{eq:buckley}  and \eqref{eq:Petermichl}  to derive the sharpness without building any particular example for each operator. 

In addition, we can observe that this is a sort of template suitable for any operator. Indeed, as an application we can derive the optimal exponent that one could expect in a weighted estimate for a maximal operator associated to a generic Muckenhoupt basis. 
Note that in this latter case it is not even possible to have an example working for a general basis. However, our method allow us to avoid the use of examples and deal with all the bases at once. 

We also obtain new results for a class of maximal functions defined in terms of Orlicz averages. For $\Phi_\lambda(t)=t\log(e+t)^\lambda$, $\lambda \in [0,\infty)$, we prove new weighted estimates for the Orlicz maximal operator $M_{\Phi_\lambda}$ which, in addition, are sharp 
as a consequence of Theorem \ref{thm:AbstractBuckley}. We remit to Section \ref{sec:orlicz} for the precise definitions of these operators 
which can be seen as continuous versions of the iterated Hardy-Littlewood maximal functions.  This  continuity is reflected in the exponent of the weighted estimates proved in Theorem \ref{thm:Orlicz}.  The operators $M_{\Phi_\lambda}$ are special cases of more general Orlicz maximal operators $M_{\Phi}$ introduced in \cite{perez95} to study sharp sufficient ``bump" type conditions for the so called two-weight problem for the Hardy-Littlewood maximal operator. Similar conditions were also considered in the two weight context for fractional integrals in \cite{perez94:indiana} and very recently in the context of Calder\'on--Zygmund operators \cite{Lerner:simpleA2} where it is used to solve the so called ``bump conjecture''.  The special case of $M_{\Phi_\lambda}$ was used in \cite{perez94:london}  to derive very sharp two weight estimate of the form $(w,M_{\Phi_\varepsilon}(w))$, 
\begin{equation*}
\|Tf\|_{L^{1,\infty}(w)}\leq  c_{\varepsilon,T}
\int_{\mathbb{R}^n} |f(x)|\,M_{L(\log L)^{\epsilon}} (w)(x)\,dx,
\quad w\geq 0.
\end{equation*}
where $T$ is any Calder\'on--Zygmund operator and where $\varepsilon>0$  is arbitrarily small. Similar sharp estimates  where also obtained in  the case $p>1$.

Even in the case where it is not known a sharp weighted estimate, we obtain a lower bound for the exponent of the $A_p$ constant. This is the case of Bochner-Riesz multipliers treated in Section \ref{sec:CZ}, Corollary \ref{cor:BR}.

\subsection{Outline}
This article is organized as follows. In Section \ref{sec:proofs} we prove the main result. Then, in Section \ref{sec:CZ} we show how to derive the sharpness of some known weighted estimates for Calder\'on--Zygmund operators with large kernels. We also exhibit lower bounds for the optimal exponent in the case of Bochner--Riesz multipliers. In Section \ref{sec:maximal-and-square} we study maximal type operators and dyadic square functions. In Section \ref{sec:fractional} we obtain results for fractional integral operators by using similar ideas and off-diagonal extrapolation techniques. Finally, in Section \ref{sec:muckenhoupt-bases} our method is used to obtain optimal exponents in the case of maximal functions defined over general Muckenhoupt bases.

\section{Proof of Theorem \ref{thm:AbstractBuckley} }\label{sec:proofs}

We present here the proof of the main results. The key tool is the Rubio de Francia's iteration scheme or algorithm to produce $A_1$ weights with a precise control of the constant  of the weight and the main underlying idea comes from extrapolation theory. 
The same ideas that we use here were already used to prove sharp weighted estimates for the Hilbert transform with $A_1$ weights in \cite{Fefferman-Pipher}. A more precise and general version was obtained recently in  \cite{Javi-Duo-JFA}. We remark that the first part of the proof, namely the proof of inequality \eqref{eq:CF} below, is a consequence of the extrapolation result from \cite{Javi-Duo-JFA} (see Theorem 3.1, first inequality of (3.2), p. 1889). We choose to include the proof for the sake of completeness. For our inequality \eqref{eq:CF-dual}, which is the analogue for large $p$, we perform a slightly different proof.

\begin{proof}[Proof of Theorem \ref{thm:AbstractBuckley}] 
We first consider the bound $\beta\ge\frac{\alpha_T}{p_0-1}$. The first step is to prove the following inequality, which can be seen as an unweighted Coifman-Fefferman type inequality relating the operator $T$ to the Hardy--Littlewood maximal function. We have that
\begin{equation} \label{eq:CF}
\|T\|_{L^{p}(\mathbb{R}^n) } \leq c\,   \|M\|_{L^{p}(\mathbb{R}^n) }   ^{\beta(p_0-p)}   \qquad  1<p<p_0.
\end{equation}

Lets start by defining, for $1<p<p_0$,  the operator $R$ as follows: 
\begin{equation*}
R(h)= \sum_{k=0}^\infty \frac1{2^k}\frac{M^k
(h)}{\|M\|_{L^{p}(\mathbb{R}^n)}^k}
\end{equation*} 
Then we have

(A) \quad $h\le R(h)$

\vspace{.2cm}

(B) \quad $\|R(h)\|_{L^{p}(\mathbb{R}^n)}\le
2\,\|h\|_{L^{p}(\mathbb{R}^n)}$

\vspace{.2cm}

(C) \quad  $[R(h)]_{A_{1}}\leq 2\,   \|M\|_{L^{p}(\mathbb{R}^n) }$

\

To verify \eqref{eq:CF}, consider $1<p<p_0$ and apply Holder's inequality to obtain
\begin{eqnarray*}
\|T(f)\|_{L^{p}(\mathbb{R}^n)} & = &  \Big(  \int_{\mathbb{R}^n} |Tf|^{p}\,  (Rf)^{-(p_{0}-p) \frac{p}{p_0}}\,(Rf)^{(p_{0}-p) \frac{p}{p_0}}\,dx  \Big)^{1/p}\\
& \le & \Big(  \int_{\mathbb{R}^n} |Tf|^{p_{0}}\,  (Rf)^{-(p_{0}-p) }\,dx  \Big)^{1/p_{0}}\, 
\Big(  \int_{\mathbb{R}^n} (Rf)^{p}\,dx  \Big)^{\frac{p_{0}-p}{pp_{0}}}\\
\end{eqnarray*}
For clarity in the exposition, we denote $w:=(Rf)^{-(p_{0}-p)}$. Then, by the key  hypothesis \eqref{eq:weighted} together with properties $(A)$ and $(B)$ of the Rubio de Francia's algorithm, we have that
\begin{eqnarray*}
\|T(f)\|_{L^{p}(\mathbb{R}^n)} &  \le & c\, [w]_{A_{p_{0}}}^{\beta}\, \Big(  \int_{\mathbb{R}^n} |f|^{p_{0}}\,  w\,dx  \Big)^{1/p_{0}}\|f\|_{L^{p}(\mathbb{R}^n) }^{\frac{p_{0}-p}{p_{0}}}\\
&\leq & c\,  [w]_{A_{p_{0}}}^{\beta}\, \Big(  \int_{\mathbb{R}^n} |f|^{p}\,  dx  \Big)^{1/p_{0}}
\|f\|_{L^{p}(\mathbb{R}^n) }^{1- \frac{p}{p_{0}}}\\
& = &  c\, [w]_{A_{p_{0}}}^{  \beta}\,
\|f\|_{L^{p}(\mathbb{R}^n) }\\
& = &   c\, [w^{1-p_0'}]_{A_{p'_{0}}}^{\beta(p_0-1)}
\|f\|_{L^{p}(\mathbb{R}^n) }  
\end{eqnarray*}
since $[w]_{A_q}= [w^{1-q'}]^{q-1}_{A_{q'}}$.  Now, since \, $\frac{p_0-p}{p_0-1}<1$ we can use Jensen's inequality to compute the constant of the weight as follows
\begin{equation*}
 [w^{1-p_0'}]_{A_{p'_{0}}}=[(Rf)^\frac{p_0-p}{p_0-1}]_{A_{p'_{0}}}
\le [R(f)]_{A_{p'_{0}}}^{\frac{p_0-p}{p_0-1}  }
\le [R(f)]_{ A_{ 1} }^{\frac{p_0-p}{p_0-1}  }
\end{equation*} 
Finally, by making use of property (C), we conclude that  
\begin{equation*}
 \|T(f)\|_{L^{p}(\mathbb{R}^n)} \le c\,   \|M\|_{L^{p}(\mathbb{R}^n) }   ^{   \beta (p_0-p) }\,
\|f\|_{L^{p}(\mathbb{R}^n) }, 
\end{equation*}
which clearly implies \eqref{eq:CF}.
Once we have proved the key inequality \eqref{eq:CF}, we can relate the exponent on the weighted estimate to the endpoint order of $T$. To that end, we will use the known asymptotic behaviour of the unweighted $L^p$ norm of the maximal function. 
It is well known that when $p$ is close to $1$, there is a dimensional constant $c$ such that
\begin{equation}\label{eq:maximal-pto1}
\| M \|_{L^{p}(\mathbb{R}^n)} \leq c\,
\frac{1}{p-1}.
\end{equation}
Then, for $p$ close to 1, we obtain
\begin{equation}
\|T\|_{L^{p}(\mathbb{R}^n) }  \le   c\, (p-1)^{-\beta (p_0-p)} \le c\, (p-1)^{-\beta (p_0-1)} 
\end{equation}
Therefore, multiplying by  $(p-1)^{\alpha_T-\varepsilon}$, using the definition of $\alpha_T$  and taking upper limits we have,
\begin{equation*}
+\infty=\limsup_{p\to1}\, (p-1)^{\alpha_T-\varepsilon}\|T\|_{L^{p}(\mathbb{R}^n) }\le  c\,\limsup_{p\to1}\, (p-1)^{\alpha_T-\varepsilon-\beta(p_0-1)}. 
\end{equation*}
This last inequality implies that $\beta\ge \frac{\alpha_T}{p_0-1}$, so we conclude the first part of the proof of the theorem. 

For the proof of the other inequality, $\beta\ge\gamma_T$, we follow the same line of ideas, but with a twist  
involving the dual space $L^{p'}(\mathbb{R}^n)$. Fix $p$, $p>p_0$. We perform the iteration technique $R'$ as before  changing $p$ with $p'$: 
\begin{equation*}
R'(h)= \sum_{k=0}^\infty \frac1{2^k}\frac{M^k
(h)}{\|M\|_{L^{p'}(\mathbb{R}^n)}^k}
\end{equation*}
Then we have

(A') \quad $h\le R'(h)$

\vspace{.2cm}

(B') \quad $\|R'(h)\|_{L^{p'}(\mathbb{R}^n)}\le
2\,\|h\|_{L^{p'}(\mathbb{R}^n)}$

\vspace{.2cm}

(C') \quad  $[R'(h)]_{A_{1}}\leq 2\,   \|M\|_{L^{p'}(\mathbb{R}^n) }$

\vspace{.2cm}

Fix $f\in L^p(\mathbb{R}^n)$. By duality there exists a non-negative
function $h\in L^{p'}(\mathbb{R}^n)$, $\|h\|_{L^{p'}(\mathbb{R}^n)}=1$, such that,
\begin{eqnarray*}
\|Tf\|_{L^p(\mathbb{R}^n)} & = & \int_{\mathbb{R}^n} |Tf(x)| h(x)\,dx\\
& \le &\int_{\mathbb{R}^n} |Tf| (R' h)^{  \frac{p-p_0}{p_0(p-1)}   } \,h^{ \frac{p(p_0-1)}{ p_0(p-1) }  } \,dx\\
& \le &\left(\int_{\mathbb{R}^n} |Tf|^{p_0} (R' h)^{  \frac{p-p_0}{p-1}   }\,dx\right)^{1/p_0} \left(\int_{\mathbb{R}^n}  h^{p'}\,dx\right)^{1/p_0'} \\
& = &   \left(\int_{\mathbb{R}^n} |Tf|^{p_0} (R' h)^{  \frac{p-p_0}{p-1}   }\,dx\right)^{1/p_0}. 
\end{eqnarray*}
Now we use the key hypothesis \eqref{eq:weighted} and H\"older's inequality to obtain
\begin{eqnarray*}
\|Tf\|_{L^p(\mathbb{R}^n)} & \le & c\, [(R' h)^{  \frac{p-p_0}{p-1}   } ]_{A_{p_0}} ^{\beta}
 \left(\int_{\mathbb{R}^n} |f|^{p_0} (R' h)^{  \frac{p-p_0}{p-1}   }\, dx\right)^{1/p_0}  \\
& \le & c\, [(R' h)^{  \frac{p-p_0}{p-1}   } ]_{A_{p_0}} ^{\beta}
  \left(\int_{\mathbb{R}^n} |f|^{p} dx\right)^{1/p}   \left(\int_{\mathbb{R}^n} (R' h)^{ p'   }\,dx\right)^{\frac{1}{p'}\frac{p-p_0}{ p_0(p-1) }   } \\
& \le & c\, [(R' h)^{  \frac{p-p_0}{p-1}   } ]_{A_{p_0}} ^{\beta}
  \left(\int_{\mathbb{R}^n} |f|^{p} dx\right)^{1/p}   \qquad \mbox{by (B') }.\\
& \le & c\, [R' h ]_{A_{p_0}} ^{\beta  \frac{p-p_0}{p-1}   }
  \left(\int_{\mathbb{R}^n} |f|^{p} dx\right)^{1/p}   \qquad \mbox{by Jensen's }\\
& \le & c\, \|M\|_{ L^{p'}(\mathbb{R}^n) }^{\beta \frac{p-p_0}{p-1}}
  \left(\int_{\mathbb{R}^n} |f|^{p} dx\right)^{1/p}   \quad \mbox{by (C')}.
\end{eqnarray*}
Hence,
\begin{equation}\label{eq:CF-dual}
\|T\|_{L^{p}(\mathbb{R}^n) } \leq c\,   \|M\|_{L^{p'}(\mathbb{R}^n) }^{\beta \frac{p-p_0}{p-1}}  \qquad  p>p_0.
\end{equation}
This estimate is similar and, somehow dual, to \eqref{eq:CF}. To finish the proof we recall that, for large $p$, namely $p>p_1>p_0$,  we have the asymptotic estimate, $\| M \|_{p'} \approx  \frac{1}{p'-1}\le p$. Therefore, we have that 
\begin{equation*}
\|T\|_{L^{p}(\mathbb{R}^n) } \le c\, p ^{\beta\frac{p-p_0}{p-1}}\le c\, p^\beta 
\end{equation*}
since $p>p_1>p_0>1$. As before, dividing by  $p^{\gamma_T-\varepsilon}$ and taking upper limits, we obtain
\begin{equation*}
+\infty=\limsup_{p\to \infty}\, \frac{\|T\|_{L^{p}(\mathbb{R}^n) }}{p^{\gamma_T-\varepsilon}}\le c\, \limsup_{p\to  \infty} \, p^{\beta-\gamma_T+\varepsilon}.
\end{equation*}
This last inequality implies that $\beta\ge \gamma_T$, so we conclude the proof of the theorem. 

\end{proof}

\subsection{Two remarks on sharpening the sharp bounds}\label{sec:sharper}

In the previous section, we showed how to prove sharp weighted bound avoiding the use of specific examples. We studied sharpness with respect to the power of the $A_p$  constant of the weight. 
However, there are several further improvements that can be made. First, we will consider mixed $A_p-A_\infty$ bounds in the spirit of \cite{HP} (see also \cite{HPR1}, and \cite{LM}). 
We also show refined estimates beyond the scale of power functions. 

\subsubsection{Mixed bounds} \label{subsec:mixed}
Here we address the problem of finding sharp ``mixed bounds". More precisely, it was shown in \cite{HP} that the maximal function satisfies
\begin{equation}\label{eq:buckmixed}
\|M\|_{L^{ p }(w)} \le c \, [w] ^{   \frac{1}{p}  }_{A_{ p }}  [\sigma] ^{   \frac{1}{p}  }_{A_{ \infty }} \qquad w \in A_{ p}.
\end{equation}
where $\sigma=\sigma_p=w^{1-p'}$ and where 
\begin{equation*}
 [\sigma]_{A_{\infty}}:= \sup_{Q}\frac{1}{\sigma(Q)}\int_Q M(\chi_Q \sigma)dx
\end{equation*}
is the Fujii-Wilson $A_{\infty}$'s constant which is much smaller than the usual (Hrushev) $A_{\infty}$ constant defined in terms of the exponential average. 
Estimate \eqref{eq:buckmixed} was proved in \cite{HP} and it was used  to improve the $A_2$ theorem from \cite{Hytonen:A2}. A better argument for proving \eqref{eq:buckmixed} was obtained in \cite{HPR1}. In addition, let us remark that in \cite{PR-twoweight} there is a new proof of this result which avoids completely the use of the delicate reverse H\"older property of $A_\infty$ weights.

We have the following corollary of Theorem \ref{thm:AbstractBuckley}.

\begin{corollary} \label{cor:AbstractBuckley-mixed}
Let $T$ be an operator (not necessarily linear). Let $w \in A_{p_0}$ for some $1<p_0<\infty$, and recall that $\sigma=w^{1-p_0'}$. Suppose further that 
\begin{equation}\label{eq:weighted-mixed}
\|T\|_{L^{ p_{0} }(w)} 
\le c\, [w] ^{\beta_1}_{A_{ p_0 }}  [\sigma] ^{\beta_2}_{A_{ \infty }}. 
\end{equation}
Then  $\beta_1+\frac{\beta_2}{p_0-1}\ge \max\left \{\gamma_T;\frac{\alpha_T}{p_0-1}\right \}$.
\end{corollary}
\begin{proof}
The proof of this variant reduces to a simple observation based on the duality properties of Muckenhoupt weights. More precisely, for any $A_p$ weight and any pair of positive exponents $\beta_1$ and $\beta_2$, we have that
\begin{equation*}
[w]^{\beta_1}_{A_p}[w^{1-p'}]^{\beta_2}_{A_\infty}\le
[w]^{\beta_1}_{A_p}[w^{1-p'}]^{\beta_2}_{A_{p'}}=
[w]^{\beta_1}_{A_p}[w]^{\frac{\beta_2}{p-1}}_{A_p}=
[w]^{\beta_1+\frac{\beta_2}{p-1}}_{A_p}.
\end{equation*}
By Theorem \ref{thm:AbstractBuckley}, we conclude that $\beta_1+\frac{\beta_2}{p_0-1}\ge \max\left \{\gamma_T;\frac{\alpha_T}{p_0-1}\right \}$.
\end{proof}

Note that this result implies that we cannot consider in \eqref{eq:buckmixed} smaller exponents. 
The same argument can be used to show  the sharpness of mixed bounds for C--Z operators. For a given C--Z operator $T$ satisfying some size condition on the kernel (see condition \eqref{eq:kernel} below) we have that $\alpha_T=1$. Therefore, any pair of exponents $(\beta_1,\beta_2)$ lying on the \emph{sharpness line} defined by $\beta_1(p-1)+\beta_2=1$ is sharp. We cannot replace any of them with a smaller quantity. For example, for any $T$ it is proved in \cite{HP} that
\begin{equation}\label{eq:mixedCZ1}
 \|T\|_{L^p(w)}\le c\, [w]_{A_p}^{2/p}[\sigma]_{A_\infty}^{2/p-1}
\end{equation}
for any $p\in (1,2]$ and any $w\in A_p$. We conclude that this pair of exponents is sharp although this does not mean that it is the best possible result. Indeed, by moving along the sharpness line, we can balance the exponents replacing some power of $[w]_{A_p}$ by  the corresponding power of $[w]_{A_\infty}$. Clearly, the best bounds are those involving a larger power in the $A_\infty$ fraction of the weight.

In the case of commutators, we have from \cite{HP} that, for $1<p\le 2$,
\begin{equation*}
 \|[b,T]\|_{L^p(w)}\le c\, [w]_{A_p}^{4/p}[\sigma]_{A_\infty}^{4/p-2}.
\end{equation*}
The exponent of this result is also sharp because the commutator satisfies that $\alpha_{[b,T]}=2$.

\subsection{Beyond power functions}

We start this section by recalling that from Buckley's original example one can conclude that inequality \eqref{eq:buckley} is sharp for \emph{arbitrary} perturbations.  Our method also allows to conclude 
the same perturbation result.  More precisely, suppose that for some $p_0 \in(1,\infty)$ and for some non-decreasing function $\varphi:[1,\infty) \to [0,\infty)$ such that 
\begin{equation*}
 \lim_{t\to \infty}\frac{\varphi(t)}{t^\frac{1}{p_0-1}}=0
\end{equation*}
we have that, for any $w\in A_{p_0}$,
\begin{equation*}
 \|M\|_{L^{p_0}(w)} \le c\, \varphi([w]_{A_{p_0}}).
\end{equation*}
We will show that this cannot hold. To see this, we argue as in Theorem \ref{thm:AbstractBuckley} and obtain that, for some positive constants $c_1,c_2$ and for $1<p<p_1<p_0$,
\begin{eqnarray*}
\|Mf\|_{L^{p}(\mathbb{R}^n)} &  \le & c_1\,  \varphi([Rf]_{A_{1}}^{p_0-p})
\|f\|_{L^p(\mathbb{R}^n)}\\
& \le & c_1\, \varphi(c_2(p-1)^{-(p_0-1)})\|f\|_{L^p(\mathbb{R}^n)}
\end{eqnarray*}
for any function $f\in L^p(\mathbb{R}^n)$. Since $
 \|M\|_{L^p(\mathbb{R}^n)} \geq c\frac{1}{p-1} $ for $p\to 1$, we obtain that
\begin{equation*}
(p-1)^{-1}  \le  c_1\,  \varphi(c_2(p-1)^{-(p_0-1)})
\end{equation*}
contradicting the assumption on $\varphi$. Therefore, we conclude that there is no possible such an improvement of \eqref{eq:buckley}.  

A similar argument can be used to derive an analogue result for a generic operator $T$ if it is known the precise endpoint behavior of $T$.

\section{Operators with large kernel and commutators}\label{sec:CZ}

Firstly, we address the problem of proving the sharpness of weighted estimates for Calder\'on--Zygmund operators, its commutators with BMO functions and vector valued extensions. We prove here the following corollary of our main result.

\begin{corollary}\label{cor:CZ-Commutators}
Let $T$ be a Calder\'on--Zygmund operator. Denote by $[b,T]$ the commutator with a BMO function $b$.  More generally, its $k$-iteration defined recursively by 
 \begin{equation*}
 T_b^k:=[T_b^{k-1},b],\qquad k\ge 1,
 \end{equation*}
 with  $k$ an integer.  The following weighted estimates are sharp

\begin{equation}\label{eq:CZ}
\|T\|_{L^{ p }(w)} \le c \, [w]^{ \max\{1,   \frac{1}{p-1} \}}_{A_{ p}}, \qquad w \in A_{ p}
\end{equation}

 \begin{equation}\label{eq:commutator}
 \|[b,T]\|_{L^{ p }(w)} \le c \, \|b\|_{BMO}\, [w]^{ 2\max\{1,   \frac{1}{p-1} \}}_{A_{ p}}, \qquad w \in A_{ p}
 \end{equation}

 \begin{equation}\label{eq:k-commutator}
 \|T_b^k\|_{L^{ p }(w)} \le c \,  \|b\|_{BMO}\, [w]^{ (k+1)\max\{1,   \frac{1}{p-1} \}     }_{A_{ p}}, \qquad w \in A_{ p}.
 \end{equation}
 
\end{corollary}

We also have the following application to vector--valued extensions.

\begin{corollary}\label{cor:vector-valued-CZ}
Given a C--Z operator $T$ we define as usual the vector--valued extension $\overline{T}_{q}$ as
\[
\overline{T}_{q}f(x)=
\left(
\sum_{j=1}^{\infty} |Tf_{j}(x)|^{q}
\right)^{1/q}
\]
where ${f}=\{f_j\}_{j=1}^{\infty}$ is a vector--valued function. Then the following estimate is sharp
 \begin{equation}\label{eq:CZ-vector-valued}
 \|\overline T_q(f)\|_{L^{ p }(w)} \le c\,  [w]^{ \max\{1,\frac{1}{p-1} \}}_{A_{p}}
 \left\|\overline{f}_q\right\|_{L^p(w)},  \qquad w \in A_{ p}
 \end{equation}
where $ \overline{f}_q(x)= \left(\sum_{j=1}^{\infty} |f_{j}(x)|^{q}\right)^{1/q}$.

\end{corollary}

\begin{proof}[Proof of Corollary \ref{cor:CZ-Commutators} and Corollary \ref{cor:vector-valued-CZ}]
All the previous inequalities are known to be true (see \cite{Hytonen:A2} for the case of C--Z operators and \cite{CPP} for the case of commutators). The bound in \eqref{eq:CZ-vector-valued} is a very recent result from \cite{HH} (see also \cite{scurry} for an alternative proof).
The sharpness follows immediately from Theorem \ref{thm:AbstractBuckley} if we check the appropriate values of $\alpha_T$ and $\gamma_T$ for each case.

In order to apply our Theorem \ref{thm:AbstractBuckley} here we need to exploit the bad behaviour at the endpoint. We remark here that the upper bound in \eqref{eq:endpointH} holds for any C--Z operator, but we need to focus on those operators $T$ such that the upper bound in \eqref{eq:endpointH} is attained. 
A general condition for this can be found in \cite[p. 42]{ste93}: suppose that the operator kernel $K$ of a C--Z operator $T$ on $\mathbb{R}^n$ satisfies that 
\begin{equation}\label{eq:kernel}
 |K(x,y)|\ge\frac{c}{|x-y|^n}.
\end{equation}
for some $c>0$ and if $x \neq y$. Then $T$ satisfies the same endpoint behaviour as the Hilbert transform in \eqref{eq:endpointH}:
\begin{equation}\label{eq:endpointCZ}
 \|T\|_{L^p(\mathbb{R}^n)}\sim O(\frac{1}{p-1}),
\end{equation}
which clearly implies that $\alpha_T=1$ (we can consider the Hilbert transform $H$ as a model example of this phenomenon in $\mathbb{R}$ and the Riesz transforms for $\mathbb{R}^n$, $n\ge 2$). Therefore, for any of such operators $T$, we have that $\alpha_T=1$. The same kind of arguments shows that $\gamma_T=1$ and then we conclude that \eqref{eq:CZ} is sharp. Since it is clear that the same holds for the vector valued extension, we conclude that \eqref{eq:CZ-vector-valued} is also sharp.

For the case of a commutator of $[b,T]$, if $T$ is a C--Z operator with a kernel $K$ satisfying \eqref{eq:kernel}, we have that $\alpha_{[b,T]}=\gamma_T=2$. Similarly, for the $k$-iterated commutator $T_b^k$, we have that $\alpha_{T_b^k}=\gamma_T=k$. This concludes with the proof of the corollary.
\end{proof}

As a final application of this result for large kernels, we present here the following consequence of our Theorem \ref{thm:AbstractBuckley} for the optimality of weighted estimates of Bochner-Riesz multipliers. For $\lambda>0$ and $R>0$, this operator  is defined by the formula
\begin{equation}\label{eq:def-BR}
 (B^\lambda_R f)(x)=\int_{\mathbb R^n}\left(1-\left(|\xi|/R\right)^2\right)^\lambda_+ \hat f(\xi)
e^{2\pi i \xi x}\ d\xi,
\end{equation}
where $\hat f$ denotes the Fourier transform of $f$. 
For $R=1$ we write simply $B^\lambda$. It is a known fact that this operator has a kernel $K_\lambda(x)$ defined by
\begin{equation}\label{eq:kernel-BR}
K_\lambda(x)= \frac{\Gamma(\lambda+1)}{\pi^\lambda}\frac{J_{n/2+\lambda}(2\pi|x|)}{|x|^{n/2+\lambda}},
\end{equation} 
where $\Gamma$ is the Gamma function and $J_\eta$ is the Bessel function of integral order $\eta$ (see \cite[p.197]{GrafakosCF}).

\begin{corollary}\label{cor:BR}
Let $1<p<\infty$. Suppose further that the following estimate holds
\begin{equation}\label{eq:weighted-BR}
\|B^{(n-1)/2}\|_{L^p(w)} \le c \, [w]^{\beta}_{A_p},
\end{equation}
for any $w\in A_p$ and where the constant $c_p$ is independent of the weight. Then  $\beta\ge \max\left \{1;\frac{1}{p-1}\right \}$.
\end{corollary}
\begin{proof}
The proof is immediate once we check that the size of the Kernel satisfies \eqref{eq:kernel}. To see this, we use the known asymptotics for Bessel functions, namely
\begin{equation*}
 J_\eta(r) = O(r^{-1/2}),
\end{equation*}
(see \cite[p.338, Example 1.4 ]{ste93}). Combining this with \eqref{eq:kernel-BR}, we obtain that 
\begin{equation*}
K_{(n-1)/2}(x)=O(|x|^{-n}),
\end{equation*}
and therefore we have that $\alpha_{B^{(n-1)/2}}=\gamma_{B^{(n-1)/2}}=1$. 
\end{proof}
In particular, this result shows that the claimed norm inequality for the maximal Bochner-Riesz operator from \cite{Li-Sun} cannot hold (see also \cite{Li-Sun-corrigendum}).

\section{Maximal operators and square functions}\label{sec:maximal-and-square}
In this section we will show how to derive sharp bounds for maximal-related operators. We also include a new result for the $k$-iterated Hardy-Littlewood maximal operator. 

\subsection{Iterated maximal operator}\label{sec:iterated-maximal}

Let $k$ be any positive integer, then the $k$-th iteration of the maximal function can be defined by induction as $M^k=M(M^{k-1})$. For this operator, we have the following sharp weighted estimate.
 \begin{corollary} \label{cor:iterated-maximal}
Let $M$ be Hardy-Littlewood maximal function and let  $1<p<\infty$ and $w\in A_p$. Then 
 \begin{equation}\label{eq:k-maximal}
 \|M^k\|_{L^p(w)}\leq c\,[w]_{A_p}^{\frac{k}{p-1}}.
 \end{equation}
 and the exponent is sharp. 
 \end{corollary}
\begin{proof}
The bound follows directly by iterating Buckley's theorem \eqref{eq:buckley} and the sharpness is a consequence of the main result Theorem \ref{thm:AbstractBuckley} since it is not difficult to verify that in this case $\alpha_{M^k}=k$.
For the iterated maximal function we also have that, for large $p$,
\begin{equation*}
\|M^k\|_{L^p(\mathbb{R}^n)}\sim 1.
\end{equation*}
Therefore, we have that $\gamma_{M^k}=0$ and then \eqref{eq:k-maximal} is sharp.
\end{proof}

\subsection{Orlicz-type Maximal functions}\label{sec:orlicz}
In this section we study maximal operators defined in terms of Orlicz norms. This kind of maximal operators allows to consider some sort of intermediate operators between integer iterations of $M$. To be more precise, let us briefly recall some definitions and properties. 
A function $\Phi:[0,\infty) \rightarrow [0,\infty)$ is called a Young function if it is continuous, convex, increasing and satisfies  $\Phi(0)=0$ and $\Phi(t) \rightarrow \infty$ as $t \rightarrow \infty$. The space $L_{\Phi}$ is a Banach function space with the Luxemburg norm defined by
\[
\|f\|_{\Phi} =\inf\left\{ t >0: \int_{\mathbb{R}^n}
\Phi\left( \frac{ |f|}{t }\right) \, dx \le 1 \right\}.
\]

Given a cube $Q$, we can also define a localized Luxemburg norm on a cube Q as
\begin{equation*}
\|f\|_{\Phi,Q}= \inf\left\{t  >0:
\frac{1}{|Q|}\int_{Q} \Phi\left(\frac{ |f|}{ t }\right)  \,
dx \le 1\right\}.
\end{equation*} 
The  corresponding maximal function is 
\begin{equation}\label{eq:maximaltype}
M_{\Phi}f(x)= \sup_{x\in Q} \|f\|_{\Phi,Q}.
\end{equation}

We are interested here in the logarithmic scale given by the functions $\Phi_\lambda(t):=t\log^\lambda(e+t)$, $\lambda \in [0,\infty)$. Note that the case $\lambda=0$ corresponds to $M$. The case $\lambda=k\in \mathbb N$ corresponds to $M_{L(\log L)^{k}}$, which is pointwise comparable to $M^{k+1}$ (see, for example, \cite{perez95:JFA}). 
For noninteger values of $\lambda$, we denote by $M_{\Phi_\lambda}=M_{L(\log L)^\lambda }$ the associated maximal operator. 
By Corollary \ref{cor:iterated-maximal}, we have that the sharp exponent in weighted estimates for these operators is $1/(p-1)$ for $\lambda=0$ and $k/(p-1)$ for $\lambda=k\in\mathbb N$. The following theorem provides a sharp bound for these intermediate exponents in $\mathbb R_{+} \setminus \mathbb N$.

\begin{theorem} \label{thm:Orlicz}
Let $\lambda>0$, $1<p<\infty$ and $w\in A_p$. Then 
 \begin{equation}\label{eq:orlicz}
 \|M_{\Phi_\lambda}\|_{L^p(w)}\leq c\, [w]_{A_p}^{\frac{1}{p}}[\sigma]_{A_\infty}^{\frac{1}{p}+\lambda}   \leq c\, [w]_{A_p}^{\frac{1+\lambda}{p-1}},
 \end{equation}
 where $\sigma=w^{1-p'}$. Furthermore, the exponents are sharp. 
 \end{theorem}
\begin{proof}
We start with the following variant of the classical Fefferman-Stein inequality which holds for any weight $w$. For $t>0$ and any nonnegative function $f$, we have that
\begin{equation}\label{eq:FeffStein-MPhi}
 w\left(\left\{x\in \mathbb{R}^n: M_{\Phi_\lambda}f(x)>t \right\}\right)\le c\int_{\mathbb {R}^n}
\Phi_\lambda\left(\frac{f(x)}{t}\right)\, Mw(x)\ dx,
\end{equation}
where is $M$ is the usual Hardy--Littlewood maximal function and $c$ is a constant independent of the weight $w$. The result 
can be obtained using a Calder\'on--Zygmund decomposition adapted to $M_{\Phi_\lambda}$ as in Lemma 4.1 from \cite{perez95}. We leave the details for the interested reader.

Now, if the weight $w$ is in $A_1$, then inequality \eqref{eq:FeffStein-MPhi} yields the linear dependence on $[w]_{A_1}$,
\begin{equation*}%\label{eq:linearA1-MPhi}
 w\left(\left\{x\in \mathbb{R}^n:M_{\Phi_\lambda}f(x)>t \right\}\right)\le c\,[w]_{A_1}\int_{\mathbb {R}^n}
\Phi_\lambda\left(\frac{f(x)}{t}\right)\, w(x)\ dx.
\end{equation*}
From this estimate and by using an extrapolation type argument as in \cite[Section 4.1]{perez-lecturenotes}, we derive easily that, for any $w\in A_p$ 
\begin{equation}\label{eq:linearAp-MPhi}
 w\left(\left\{x\in \mathbb{R}^n:M_{\Phi_\lambda}f(x)>t \right\}\right)\le c\,[w]_{A_p}\int_{\mathbb {R}^n}
\Phi_\lambda\left(\frac{f(x)}{t}\right)^p\, w(x)\ dx.
\end{equation}
Now, we follow the same ideas from \cite[Theorem 1.3]{HPR1}. We write the $L^p$ norm as
\begin{equation*}
 \|M_{\Phi_\lambda} f\|_{L^p(w)}^p  \leq  c \int_{0}^{\infty}  t^{p} w \{x\in \mathbb{R}^n:M_{\Phi_\lambda} f_t(x) > t\}
  \frac{dt}{t}
\end{equation*}
where $f_t:=f\chi_{f>t}$. Since $w\in A_p$, then by the precise open property of $A_p$ classes, we have that $w\in A_{p-\varepsilon}$ where $\varepsilon\sim \frac{1}{[\sigma]_{A_\infty}}$. Moreover, the constants satisfy that $[w]_{A_{p-\varepsilon}}\le c[w]_{A_p}$ (see \cite[Theorem 1.2]{HPR1}). We apply \eqref{eq:linearAp-MPhi} with $p-\varepsilon$ instead of $p$ to obtain after a change of variable
\begin{eqnarray*}
\|M_{\Phi_\lambda} f\|_{L^p(w)}^p & \leq & c\, [w]_{A_{p}}\int_{\mathbb{R}^n}  f^p \int_{1}^{\infty} \frac{\Phi_\lambda(t)^{p-\varepsilon}}{t^p}\frac{dt}{t}\ w \ dx \\
& \le & c\, [w]_{A_{p}}\int_1^\infty \frac{(\log(e+ t))^{p\lambda}}{t^\varepsilon}\frac{dt}{t}\ \|f\|^p_{L^p(w)}\\
& \le & c\, [w]_{A_{p}}\left(\frac{1}{\varepsilon}\right)^{\lambda p+1}\ \|f\|^p_{L^p(w)}\\
& \le & c\, [w]_{A_{p}}[\sigma]_{A_\infty}^{\lambda p+1}\ \|f\|^p_{L^p(w)}\\
\end{eqnarray*}
Takin $p$-roots we obtain the desired estimate \eqref{eq:orlicz}.

Regarding the sharpness, we will  prove now that the exponent in the last term of  \eqref{eq:orlicz} cannot be improved. This follows from Theorem \ref{thm:AbstractBuckley} since it is easy to verify that 
\begin{equation*}
 \|M_{\Phi_\lambda}\|_{L^p(\mathbb{R}^n)}\sim \frac{1}{(p-1)^{1+\lambda}}.
\end{equation*}
From this estimate we conclude that the endpoint order verifies $\alpha_T=1+\lambda$ for $T=M_{\Phi_\lambda}$. As a final remark, we mention that the exponents of the middle term in \eqref{eq:orlicz} are also sharp by the same argument as in Section \ref{subsec:mixed}.
\end{proof}

\subsection{Vector valued maximal functions}\label{sec:vector-valued-maximal}
We now consider the vector-valued extension of the H-L maximal function. Let $1<q<\infty$ and $1<p<\infty$, then this operator is defined as:
\begin{equation*}
\overline{M}_qf(x)=\Big( \sum_{j=1}^{\infty} (Mf_j(x))^q \Big)^{1/q},
\end{equation*}
where ${f}=\{f_j\}_{j=1}^{\infty}$ is a vector-valued function. For this operator we obtain this corollary.

\begin{corollary}\label{cor:vector-valued-maximal}
For $1<p<\infty$, and for any $w\in A_p$, the following norm inequality is sharp.
\begin{equation}\label{eq:vector-valued-maximal}
\|\overline{M}_qf \|_{L^p(w)}\le c\, [w]^{\max\{\frac{1}{q},\frac{1}{p-1}\}}\|\overline{f}_q\|_{L^p(w)}, \qquad w \in A_{ p}.
\end{equation}
\end{corollary}

\begin{proof}

The bound was proved in \cite{CMP-ADV}. For the sharpness of \eqref{eq:vector-valued-maximal}, although we cannot apply directly our main Theorem \ref{thm:AbstractBuckley}, it is easy to see that once we write the $L^p$ norm of $\left(\sum_j M(f_j)^q\right)^{1/q}$, the same arguments yield the desired result, namely, the analogue of Theorem \ref{thm:AbstractBuckley} in the vector-valued setting. Therefore, the sharpness will follow if we check the values of $\alpha_{\overline{M}_q}$ and $\gamma_{\overline{M}_q}$. The fact that $\alpha_{\overline{M}_q}=1$ can be verified in the same way as in the case $q=1$. For $\gamma_{\overline{M}_q}$, we can find an example of a vector- valued function satisfying $\|\overline{M}_qf\|_{L^p}\ge c p^{1/q}\|f\|_{L_{\ell^q}^p}$ which implies that $\gamma_{\overline{M}_q}=1/q$. This was already known; see \cite[p.75]{ste93} for the classic proof.
\end{proof}

\subsection{Square functions}\label{sec:square}

We include here the case of the dyadic square function $S_d$, since it behaves similarly to the vector--valued maximal function. It is defined as follows. Let $\Delta$ denote the collection of dyadic cubes in $\mathbb{R}^n$. Given $Q\in\Delta$, let $\hat Q$ be its dyadic parent, that is, the unique dyadic cube containing $Q$ whose side-length is twice that of $Q$. Then, the dyadic square function is the operator 
\begin{equation*}
S_df(x) = \left(\sum_{Q\in\Delta} (f_Q-f_{\hat Q})^2\chi_Q(x)\right)^{1/2} 
\end{equation*}
where $f_Q = \avgint_Qf(x)\ dx$. 

For this operator the result is the following corollary of Theorem \ref{thm:AbstractBuckley}.

\begin{corollary}\label{cor:square-dyadic}
For $1<p<\infty$, and for any $w\in A_p$, the following norm inequality is sharp:
\begin{equation}\label{eq:square-dyadic}
\|S_df\|_{L^p(w)}\le c\,[w]^{\max\{\frac{1}{2},\frac{1}{p-1}\}}\|f\|_{L^p(w)}.
\end{equation} 
\end{corollary}

\begin{proof}
Again, the inequality is known to be true (see \cite{CMP-ADV} and references therein). For the sharpness, we just check the values of the two endpoint orders. We first note that $\alpha_{S_d}=1$ by looking at the indicator function of the unit cube (as in the case of the maximal function). The value of $\gamma_{S_d}=\frac{1}{2}$ was previously known, see for instance \cite[p. 434]{CMP-ADV}. In particular, there is an explicit example of a function $f$ such that 
$\|S_d f\|_{L^p}\ge c p^{1/2}\|f\|_{L^p}$.
It should be mentioned that the case $p\to\infty$ was already implicitly considered in \cite{Fefferman-Pipher}.
\end{proof}

We remark that this type of arguments for the sharpness of weighted estimates were already in the cited article \cite{CMP-ADV} for the square function and the vector-valued maximal function. However, this was used only for these two cases and only for large values of $p$.

\section{Fractional integral operators }\label{sec:fractional}

In the same spirit as in the previous sections, we can prove the sharpness of weighted estimates for fractional integral operators. For $0<\alpha<n$, the fractional integral operator or Riesz
potential  $I_\alpha$ is defined by
\begin{equation*}
I_\alpha f(x)=\int_{R^n} \frac{f(y)}{|x-y|^{n-\alpha}}dy.  
\end{equation*}
We also consider the related fractional maximal operator $M_\alpha$ given by
\begin{equation*}
M_\alpha f(x)=\sup_{Q\ni x} \frac{1}{|Q|^{1-\alpha/n}}\int_Q |f(y)| \ dy. 
\end{equation*}
It is well known (see \cite{MW-fractional}) that these operators are bounded from $L^p(w^p)$ to $L^q(w^q)$ if and only if the exponents $p$ and $q$ are related by the equation $1/q-1/p=\alpha/n$ and $w$ satisfies the so called $A_{p,q}$ condition. More precisely, $w\in A_{p,q}$ if 
 \begin{equation*}
[w]_{A_{p,q}}\equiv \sup_Q\left(\frac{1}{|Q|}\int_Q w^q \ dx\right)\left(\frac{1}{|Q|}\int_Q w^{-p'}\ dx\right)^{q/p'}<\infty.   
 \end{equation*}

An extrapolation theorem  for these classes of weights, often called off-diagonal extrapolation theorem, was obtained for the first time by  Harboure, Mac\'ias and Segovia in \cite{HMS} although we will use 
a new version from \cite{Javi-Duo-JFA}.

We have the following proposition.
\begin{proposition}
Suppose that $0\leq \alpha <n$, $1<p<n/\alpha$ and $q$ is defined by the relationship $1/q=1/p-\alpha/n$.  If $w\in A_{p,q}$,  then the following inequalities are sharp
\begin{equation}\label{eq:frac-maximal}
\|M_\alpha \|_{L^p(w^p) \to L^q(w^q)} \leq c\,
[w]_{A_{p,q}}^{\frac{p'}{q}(1-\frac{\alpha}{n})}. 
\end{equation}
and 
\begin{equation}\label{eq:frac-integral}
\|I_\alpha\|_{L^p(w^p) \to L^q(w^q)}\leq
c\,[w]_{A_{p,q}}^{(1-\frac{\alpha}{n})\max\{1,\frac{p'}{q}\}}. 
\end{equation}
\end{proposition}
\begin{proof}
Both inequalities are known to be true and the proof can be found in \cite{LMPT}. There, it was also proved the sharpness by constructing appropriate examples. We show here that we can derive the sharpness by using a version of our approach adapted to the setting of off diagonal extrapolation. Let $1< p_0<\infty$ and $0<q_0<\infty$ such that $1/p_0-1/q_0=\alpha/n$. Suppose that we have, for some $\beta>0$, the following inequality.
\begin{equation*}
\|M_\alpha \|_{L^{p_0}(w^{p_0}) \to L^{q_0}(w^{q_0})} \leq c\,
[w]_{A_{p_0,q_0}}^\beta,
\end{equation*}
for any $w\in A_{p,q}$. We apply Theorem 5.1 from \cite{Javi-Duo-JFA} to obtain, for any $\frac{n}{n-\alpha}<q<q_0$, the unweighted estimate
\begin{equation}\label{eq:fract-unweighted}
\|M_\alpha f \|_{ L^q(\mathbb{R}^n)} \leq c\, \|M\|_{L^{q\frac{n-\alpha}{n}}(\mathbb{R}^n)}^{\beta(q_0-q)\frac{n-\alpha}{n}}\|f \|_{ L^p(\mathbb{R}^n)} 
\end{equation}
where $M$ is the usual H--L maximal operator. Now we need to use the analogue of the endpoint order for the fractional maximal operator. From \eqref{eq:fract-unweighted} we can derive the following inequality:
\begin{equation}\label{eq:fract-orders}
\left(q-\frac{n}{n-\alpha}\right)^{-1/q} \leq c\, \left(q-\frac{n}{n-\alpha}\right)^{-\beta(q_0-q)\frac{n-\alpha}{n}}. 
\end{equation}
This can be done by estimating the operator norm of the fractional maximal operator. 
On the left hand side of \eqref{eq:fract-orders} we used the fact that 
\begin{equation*}
\|M_\alpha\|^q_{L^p(\mathbb{R}^n)\to L^q(\mathbb{R}^n)}\ge \frac{1}{q-\frac{n}{n-\alpha}} 
\end{equation*}
On the right hand side of \eqref{eq:fract-orders} we just use again that $\|M\|_{L^r}\sim 1/(r-1)$ for $r$ close to 1. Arguing as before, if we let $q$ go to the critical value $\frac{n}{n-\alpha}$ we obtain that
\begin{equation*}
 \beta\ge (1-\alpha/n)\frac{1}{q_0\frac{n-\alpha}{n}-1}=(1-\alpha/n)\frac{p'_0}{q_0}
\end{equation*}

The sharpness for the case of the fractional integral, namely inequality \eqref{eq:frac-integral}, follows essentially the same steps. We need to prove that the inequality
\begin{equation*}
\|I_\alpha\|_{L^p(w^p) \to L^q(w^q)}\leq
c\,[w]_{A_{p,q}}^\beta, \qquad w\in A_{p,q}
\end{equation*}
implies that $\beta\ge (1-\frac{\alpha}{n})\max\{1,\frac{p'}{q}\}$. For the bound $\beta\ge(1-\frac{\alpha}{n})\frac{p'}{q}$ we can repeat the previous proof, since the fractional integral also satisfies that 
\begin{equation*}
\|I_\alpha\|^q_{L^p(\mathbb{R}^n)\to L^q(\mathbb{R}^n)}\ge \frac{1}{q-\frac{n}{n-\alpha}} 
\end{equation*}
The other case, namely $\beta\ge (1-\frac{\alpha}{n})$ follows easily by duality. We left the details for the interested reader.
\end{proof}

\section{Muckenhoupt bases}\label{sec:muckenhoupt-bases}

In this section we address the problem of finding optimal exponents for maximal operators defined over Muckenhoupt bases. Recall that  given a family $\mathcal{B}$ of open sets, we can define the maximal operator $M_\mathcal{B}$ as 
\begin{equation*}
 M_\mathcal{B}f(x)=\sup_{x\in B\in \mathcal{B}}\avgint_B |f(y)| \ dy,
\end{equation*}
if $x$ belongs to some set $b\in \mathcal{B}$ and $M_\mathcal{B}f(x)=0$ otherwise. The natural classes of weights associated to this operator are defined in the same way as the classical Muckenhoupt classes: $w\in A_{p,\mathcal{B}}$ if
\begin{equation*}
[w]_{A_{p,\mathcal{B}}}:=\sup_{B\in\mathcal{B}}\left(\frac{1}{|B|}\int_{B}w(y)\ dy \right)\left(\frac{1}{|B|}\int_{B}w(y)^{1-p'}\ dy \right)^{p-1}<\infty.
\end{equation*}

We say that a basis $\mathcal{B}$ is a Muckenhoupt basis if $M_\mathcal{B}$ is bounded on $L^p(w)$ whenever $w\in A_{p,\mathcal{B}}$ (see \cite{perez-pubmat}).

In this generality, we also can prove a lower bound for the best possible exponent in a weighted estimate. The only requirement on the operator $M_\mathcal{B}$ is that its $L^p$ norm must blow up when $p$ goes to 1 (no matter the ratio of blow up). Precisely, we have the following theorem.

\begin{theorem}\label{thm:muckenhoupt-bases}
Let $\mathcal{B}$ be a Muckenhoupt basis. Suppose in addition that the associated maximal operator $M_\mathcal{B}$ satisfies the following weighted estimate:
\begin{equation}
 \|M_\mathcal{B}\|_{L^{p_0}(w)}\leq c\, [w]_{A_{p_0,\mathcal{B}}}^{\beta}.
\end{equation}
If $\displaystyle\limsup_{p\to 1^+}\|M_\mathcal{B}\|_{L^p(\mathbb{R}^n)}=+\infty$, then $\beta\ge \frac{1}{p_0-1}$.
\end{theorem}

\begin{proof}
The idea is to perform the iteration technique from Theorem \ref{thm:AbstractBuckley} but with $M_\mathcal{B}$ instead of the standard H--L maximal operator. Then we obtain, for $ 1<p<p_0$, that
\begin{equation} \label{eq:CF-muckenhoupt-bases}
\|M_\mathcal{B}\|_{L^{p}(\mathbb{R}^n) } \leq c\, \|M_\mathcal{B}\|_{L^{p}(\mathbb{R}^n) }   ^{\beta(p_0-p)}\le c\, \|M_\mathcal{B}\|_{L^{p}(\mathbb{R}^n) }   ^{\beta(p_0-1)}.
\end{equation}
The last inequality holds since $\|M_\mathcal{B}\|_{L^{p}(\mathbb{R}^n) }\ge 1$. 
We remark here that, since we are comparing $M_\mathcal{B}$ to itself, it is irrelevant to know the precise quantitative behaviour of its $L^p$ for $p$ close to 1. In fact, we cannot use any estimate like \eqref{eq:maximal-pto1} since we are dealing with a generic basis. Just knowing that the $L^p$ norm blows up when $p$ goes to 1, allows us to conclude that $\beta\ge \frac{1}{p_0-1}$. 
\end{proof}
As an example of this result, we can show that the result for Calder\'on weights from \cite{DMRO-calderon} is sharp. Precisely, for the basis $\mathcal{B}_0$ of open sets in $\mathbb{R}$ of the form $(0,b)$, $b>0$, the authors prove that the associated maximal operator $N$ defined as
\begin{equation*}
 Nf(t)=\sup_{b>t}\frac{1}{b}\int_0^b |f(x)|\ dx
\end{equation*}
is bounded on $L^p(w)$ if and only if $w\in A_{p,\mathcal{B}_0}$ and, moreover, that 
\begin{equation*}
 \|N\|_{L^p(w)}\le c\, [w]_{A_{p,\mathcal{B}_0}}^{\frac{1}{p-1}}.
\end{equation*}
By the preceding result, this inequality is sharp with respect to the exponent on the characteristic of the weight.

We can also apply Theorem  \ref{thm:muckenhoupt-bases} to the basis of rectangles in $\mathbb R^n$ with sides parallel to the coordinate axes. We detail this case in the following subsection.

\subsection{The strong maximal function}\label{sec:strong}

All the sharp results we have obtained here concern the classical or one--parameter theory, where the operators commute with one-parameter dilations of $\mathbb{R}^n$. A natural question would be to study this kind of sharp quantitative estimates for {\em multi--parameter} operators. As a first step we have tried to apply the approach we have presented here to the most basic example of the multiparameter theory, that is the strong maximal function. However, we have not obtained a satisfactory answer even for this. 

Let us recall first some definitions and known estimates to understand why our {\emph{template}} does not work for this operator. For a locally integrable function $f$ on $\mathbb{R}^n$ we will denote by $M_s f$ the strong maximal function:
\begin{align*}
 M_sf(x)= \sup_{\substack{R\ni x}} \frac{1}{|R|} \int_R |f(y)|dy,\quad x\in\mathbb R^n,
\end{align*}
where the supremum is taken over all the rectangles in $\mathbb R^n$ with sides parallel to the coordinate axes. This operator is bounded in $L^p(\mathbb{R}^n)$. Indeed,
\begin{equation}\label{eq:strong R}
\| M_s \|_p \approx  
(p')^n
\end{equation}
where $1<p<\infty$. We will say that $w$ belongs to the class $A_p ^*$, $1<p<\infty$, whenever
\begin{align*}
 [w]_{A_p ^*}=\sup_{R} \bigg(\frac{1}{|R|}\int_R w \bigg)  \bigg( \frac{1}{|R|}\int_R w^{1-p'} \bigg)^{p-1}<+\infty
\end{align*}
where the supremum is taken over all the rectangles in $\mathbb R^n$ with sides parallel to the coordinate axes. Thus $A_p^*$ is the class of weights associated naturally  with $n$--dimensional intervals. As it happened with the Hardy-Littlewood maximal function, this class of weights characterizes completely the boundedness of the strong maximal function in weighted Lebesgue spaces. In fact, it is not difficult to see that
\begin{equation}\label{eq:strong}
\|M_s\|_{L^{ p }(w)} \le c \, [w]^{   \frac{n}{p-1}  }_{A_{ p}^*}, \qquad w \in A_{ p}^*.
\end{equation}
To study which would be the sharp exponent in the last inequality, we could reproduce the proof of Theorem \ref{thm:AbstractBuckley} replacing in the Rubio de Francia's algorithm the maximal function by the strong maximal one and making a suitable use of estimate \eqref{eq:strong R}. The analogue for the multiparameter setting is the following result. Suppose that a given operator $T$, bounded in $L^p(\mathbb{R}^n)$ for $1<p<\infty$, satisfies a weighted inequality like 
\begin{equation*}\label{eq:strong-weighted}
 \|T\|_{L^p(w)}\le c\, [w]_{A^*_p}^\beta
\end{equation*}
for any $w\in A_p^*$. In addition, define the endpoint order $\alpha_T$ as before. Then, the same arguments from Theorem \ref{thm:AbstractBuckley} allow us to conclude that 
any exponent in \eqref{eq:strong-weighted} needs to be
\begin{equation*}
\beta\geq\frac{\alpha_{T}}{n(p-1)}.
\end{equation*}
Going back to the case of the strong maximal function, we have that $\alpha_{M_s}=n$ according to Definition \ref{def:orders} and the estimate \eqref{eq:strong R}. Therefore we just obtain a trivial estimate.

\section{Acknowledgements}
We are deeply in debt to Javier Duoandikoetxea for many valuable comments and suggestions on this problem. In particular, he pointed out and brought to our attention the application of our results to  Bochner-Riesz multipliers and Muckenhoupt bases.

The first author is supported by the Spanish Ministry of Science and Innovation grant MTM2012-30748,
the second and third authors are also supported by the Junta de Andaluc\'ia, grant FQM-4745.

\bibliographystyle{alpha}

\end{document}